\documentclass{article}

\usepackage{amsmath, amssymb, amsthm, graphicx, appendix}

\newtheorem{thm}{Theorem}
\newtheorem{conj}{Conjecture}
\newtheorem{lem}{Lemma}

\title{Sum-free cyclic multi-bases and
constructions of Ramsey algebras}
\author{Jeremy F. Alm\footnote{Corresponding author.}\\ Illinois College\\ Jacksonville, IL 62650\\ \texttt{alm.academic@gmail.com} \and Jacob Manske\\ Texas State University\\ San Marcos, TX 78666\\ \texttt{jmanske@gmail.com}}
\begin{document}
\maketitle

\renewcommand{\thefootnote}{\fnsymbol{footnote}}
\footnotetext{2010 MSC codes 05E15, 03G15}
\renewcommand{\thefootnote}{\arabic{footnote}}


\begin{abstract}
Given $X\subseteq \mathbb{Z}_N$, $X$ is called a \emph{cyclic basis}
if $(X+X)\cup X=\mathbb{Z}_N$, \emph{symmetric}  if $x\in X$ implies
$-x \in X$, and \emph{sum-free} if $(X+X)\cap X=\varnothing$. We
ask, for which $m$, $N\in\mathbb{Z}^+$ can the set of non-identity
elements of $\mathbb{Z}_N$ be partitioned into $m$ symmetric
sum-free cyclic bases? If, in addition, we require that distinct
cyclic bases interact in a certain way, we get a proper relation
algebra called a Ramsey algebra. Ramsey algebras (which have also
been called Monk algebras) have been constructed previously for
$2\leq m\leq 7$. In this manuscript, we provide constructions of
Ramsey algebras for every positive integer $m$ with $2\leq m\leq
400$, with the exception of $m=8$ and $m=13$.
\end{abstract}

%

\section{Introduction and motivation}

Let $N$ be a positive integer,, and let $\mathbb{Z}_N$ denote the
ring of integers modulo $N$.  For $X \subseteq \mathbb{Z}_{N}$, let
\[X + X = \left\{x_{1} + x_{2} : x_{1},x_{2} \in X\right\}.\] A
subset $X \subseteq \mathbb{Z}_{N}$ is called a \emph{cyclic basis}
for $\mathbb{Z}_{N}$ if $\left(X + X\right) \cup X =
\mathbb{Z}_{N}$. A cyclic basis $X$ is called \emph{sum-free} if
$\left(X +X \right) \cap X = \varnothing$.

One interesting question about cyclic bases is how small they can
be. More precisely, let $m(2,k)$ denote the largest $N$ such that
there is some $A\subseteq \mathbb{Z}_N $ with $|A|=k$ and $A\cup
(A+A) =\mathbb{Z}_N$ (see \cite{Jia10}). It is easy to see that
$m(2,k)=O(k^2)$; an interesting question is how large the
coefficient $\alpha$ on $k^2$ can be made so that $m(2,k)\geq \alpha
k^2$. The largest $\alpha$ currently known is $1/3 - \varepsilon$
for any $\varepsilon>0$ and for all sufficiently large $k$, due to
 Shen and Jia \cite{JiaShen}. Jia's excellent manuscript
\cite{Jia10} provides a great background on the topic of cyclic
bases as well as a healthy list of references, and the authors wish
to refer the interested reader to it.

For $m \in \mathbb{Z}^{+}$, a partition of $\mathbb{Z}_{N} \setminus
\left\{0\right\}$ into sets $X_{0},X_{1},\ldots,X_{m-1}$ is called a
\emph{sum-free cyclic multi-basis} if $X_{i}$ is a sum-free cyclic
basis for $i = 0,1,\ldots,m-1$.

A subset $X \subseteq \mathbb{Z}_{N}$ is called \emph{symmetric} if  $\forall x \left(x \in X \rightarrow -x \in X\right)$; that is, $X$ is closed under additive inverse. We take a moment to note that if $X$ is symmetric, then $X  - X = X + X$.

For ease, if a partition is a symmetric sum-free cyclic multi-basis, we shall call it an \emph{SSFCMB}.

We also desire our partitions to have one more property:
\[\forall i \forall j \left(i \neq j \rightarrow X_{i} + X_{j} = \mathbb{Z}_{N}\setminus \left\{0\right\}\right).\]
If a partition has this property, we shall say that the partition satisfies the \emph{mandatory triangle condition}. (The reason for this term is due to the connection to Ramsey algebras, which will be made clear shortly.)

The smallest example of a non-trivial SSFCMB which satisfies the
mandatory triangle condition has parameters $m = 2$ and $N = 5$. The
partition of $\mathbb{Z}_{5}\setminus \left\{0\right\}$ that we use
is $X_{0} = \left\{1,4\right\}$ and $X_{1} = \left\{2,3\right\}$. We
leave it to the reader  to check that this partition is an SSFCMB
possessing the mandatory triangle condition. The next smallest
example has parameters $m = 3$ and $N = 13$. Here, the partition of
$\mathbb{Z}_{13}\setminus \left\{0\right\}$ is
\begin{eqnarray*}
X_{0} & = & \left\{1,5,8,12\right\}\\
X_{1} & = & \left\{2,3,10,11\right\},\ \text{and}\\
X_{2} & = & \left\{4,6,7,9\right\}.
\end{eqnarray*}

In \cite{comer83}, Comer constructs these previous  two examples as
well several others. The focus of this manuscript is to attack the
following question: given $m \in \mathbb{Z}^{+}$, can we find $N \in
\mathbb{Z}^{+}$ and a partition of
$\mathbb{Z}_{N}\setminus\left\{0\right\}$ into $m$ parts which is an
SSFCMB that satisfies the mandatory triangle condition? Our main
result is summarized below as Theorem \ref{mainthm}.

\begin{thm}\label{mainthm} For every positive integer $m$ with $2 \leq m \leq 400$, with the possible exception of  $m = 8$ and $m = 13$, there exists a positive integer $N$ and a partition of $\mathbb{Z}_{N}\setminus \left\{0\right\}$ which is an SSFCMB that satisfies the mandatory triangle condition.
\end{thm}

The proof of Theorem \ref{mainthm} is based on an optimized computer search, whose inception was based in the ideas of Comer and Maddux. Section \ref{description} describes our search algorithm. Section \ref{efficiency} describes how we have optimized the search. Appendix \ref{mikesandnovembers} contains the results of the search, summarized as a table of values of $m$ and $N$, with enough information to recreate the partition (for the curious reader).

Before we get to the search algorithm, we wish to discuss the
connection between SSFCMBs and the theory of relation algebras.   A
\emph{(proper) relation algebra} is an algebra $\langle
A,\cup,^c,\circ,^{-1},Id\rangle$, where $A$ is a subset of the power
set of some equivalence relation $E$ that forms a Boolean algebra
$\langle A,\cup,^c\rangle$, the operator $\circ$ is composition of
relations, the operator $^{-1}$ is conversion of relations, and $Id$
is the identity subrelation of $E$. A \emph{Ramsey algebra in $m$
colors} is a proper relation algebra where  all of the \emph{atoms}
(i.e., minimal non-empty relations) $A_0,\ldots,A_{m-1}$ distinct
from $Id$ satisfy
\begin{enumerate}
  \item $A^{-1}_i=A_i$;
  \item $A_i\circ A_i=A^c_i$;
  \item for $i\neq j$, $A_i\circ A_j=Id^c$.
\end{enumerate}

A \emph{cyclic Ramsey algebra in $m$ colors} is a \emph{Ramsey
algebra} where $E=\mathbb{Z}_N\times\mathbb{Z}_N$ for some $N\in
\mathbb{Z}^+$, and all of the atoms $A_i$ are defined by
``difference sets" $X_i$, so that $A_i=\{(x,y):x-y\in X_i\}$, where
$-X_i=X_i$.  In this case,  each $X_i\subseteq \mathbb{Z}_N$ is a
symmetric sum-free cyclic basis, since it must satisfy
$X_i+X_i=\mathbb{Z}_N\setminus X_i$. Furthermore, the collection
$X_0,\ldots,X_{m-1}$ is an SSFCMB for $\mathbb{Z}_N$ that has the
additional property that each sum $X_i+X_j$ is as large as it can
possibly be; that is,
\begin{equation}
  \forall i, X_i+X_i=\mathbb{Z}_N\setminus X_i\ \text{and} \label{eqno:1}
\end{equation}
\begin{equation}
  \forall i\neq j, X_i+X_j=\mathbb{Z}_N\setminus \{0\}.\label{eqno:2}
\end{equation}

Thus the existence of a cyclic Ramsey algebra in $m$ colors is
equivalent to the  existence of an SSFCMB in $m$ parts satisfying
(\ref{eqno:1}) and (\ref{eqno:2}).

Our example above with $m = 2$ and $N = 5$ is a cyclic Ramsey algebra in 2 colors.
Let $X_0=\{1,4\}$ and $X_1=\{2,3\}$.  Define two relations
\[R=\{(x,y)\in\mathbb{Z}_5\times\mathbb{Z}_5 : x-y\in X_0\}\]
and
\[B=\{(x,y)\in\mathbb{Z}_5\times\mathbb{Z}_5 : x-y\in X_1\}.\]

Let $R$ and $B$ be the two atoms besides the identity  $Id=\{(x,x) : x\in\mathbb{Z}_5\}$. They satisfy
\begin{align*}
  R\circ R &=B\cup Id,\\
  B\circ B &=R\cup Id, \text{ and}\\
  R\circ B &=R\cup B.
\end{align*}

Note that $R\circ R=B\cup Id$ follows from the fact that
$X_0+X_0=X_1\cup\{0\}$.  $X_0$ is sum-free,  which means that $R$ is
``triangle-free," as in the graph depicted in Figure
\ref{k5coloring}. The graph depicts the relations $R$ and $B$ as
sets of edges in $K_5$ colored red and blue, respectively.

Similarly, for  $m = 3$  and $N = 13$ we can construct a cyclic Ramsey algebra in 3 colors.  As above, let

\begin{eqnarray*}
X_{0} & = & \left\{1,5,8,12\right\}\\
X_{1} & = & \left\{2,3,10,11\right\},\ \text{and}\\
X_{2} & = & \left\{4,6,7,9\right\}.
\end{eqnarray*}
be a partition of the non-identity elements of $\mathbb{Z}_{13}$. Define three relations
\begin{eqnarray*}
R & = & \{(x,y)\in\mathbb{Z}_{13}\times\mathbb{Z}_{13} : x-y\in X_0\},\\
B & = & \{(x,y)\in\mathbb{Z}_{13}\times\mathbb{Z}_{13} : x-y\in X_1\},\ \text{and} \\
G & = & \{(x,y)\in\mathbb{Z}_{13}\times\mathbb{Z}_{13} : x-y\in X_2\}.
\end{eqnarray*}
See Figure \ref{k13coloring} for a graph that depicts  the relations
$R$, $B$, and $G$ as sets of edges in $K_{13}$ colored red, blue,
and green, respectively. These examples illustrate the fact that the
question of the existence of Ramsey algebras can be stated in purely
graph-theoretical terms.

\begin{figure}[htb]
\label{k5coloring}
\begin{center}
\includegraphics[width=2in]{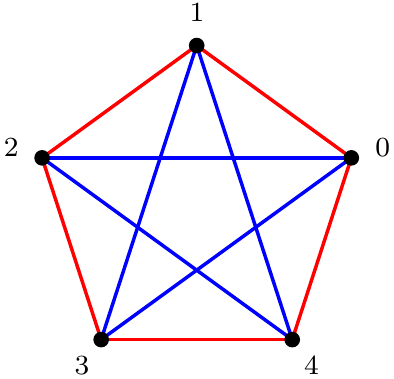}
\end{center}
\caption{A coloring of the edges of $K_{5}$ in two colors which represents a partition of $\mathbb{Z}_{5}\setminus\left\{0\right\}$ into two sets. The partition is an SSFCMB that satisfies the mandatory triangle condition.}
\end{figure}

\begin{figure}[htb]

\begin{center}
\includegraphics[width=2in]{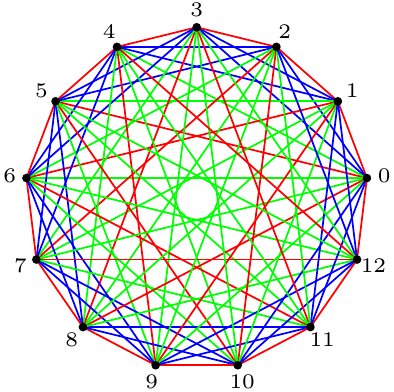}
\end{center}
\caption{A coloring of the edges of $K_{13}$ in three colors which represents a partition of $\mathbb{Z}_{13}\setminus\left\{0\right\}$ into three sets. The partition is an SSFCMB that satisfies the mandatory triangle condition.}
\label{k13coloring}
\end{figure}

This is where the name ``Ramsey algebra" comes from --- the atoms of
a (cyclic) Ramsey algebra, interpreted as edge sets in a complete
graph $K_N$ instead of as symmetric binary relations, yield an
edge-coloring of $K_N$ in $m$ colors that contains no monochromatic
triangles. Note that under this interpretation, the mandatory
triangle condition says that every edge participates in every
possible type of triangle \emph{except} for monochromatic triangles.

Let us pause briefly to discuss terminology further.  The
abstract-algebraic counterpart to Ramsey algebras has been used in
the literature under various names for over 30 years.  They were
first mentioned by Maddux \cite{Mad82} but given no name.  They have
been called, variously, Monk algebras, Maddux algebras, and very
recently, Ramsey algebras.  In \cite{HH}, Hirsch and Hodkinson use
the term ``Monk algebra" to refer to a more general kind of algebra
in which the colors can come in different ``shades", but in their
usage monochromatic triangles are still forbidden.  In his 2011 talk
at the AMS meeting in Iowa City \cite{Mad11}, Maddux defined (for
the first time, it would seem) a Ramsey algebra as we did above.  In
\cite{Kowalski}, Kowalski uses the term ``Ramsey algebra" to refer
to an abstract algebra, so that the Ramsey algebras of the present
paper would be, in his terminology, \emph{representations} of
(Kowalski's) Ramsey algebras. We choose to adopt Maddux's
terminology, since it allows the problem of existence of Ramsey
algebras to be stated in purely combinatorial terms.  See Kowalski's
paper \cite{Kowalski} for the abstract-algebraic treatment.

 The question of the existence of Ramsey algebras in all numbers of
colors was raised (though not in those terms) by Maddux in
\cite{Mad82}, Problem 2.7. Sometime in the mid-80s, Erd\H{o}s,
Szemer\'{e}di, and Trotter gave a purported proof that Ramsey
algebras exists for all sufficiently large $m $. Comer told Trotter
about the problem sometime in the early-to-mid 80s.  Trotter sent a
version of the purported proof to Comer  via e-mail, and Comer sent
it to Maddux \cite{M}.  Unfortunately, their ``proof" was in error,
as their construction did not satisfy the mandatory triangle
condition. Comer produced constructions of cyclic Ramsey algebras
for $m=2,3,4,5$ in 1983 \cite{comer83}. In 2011, Maddux produced
constructions for $m=6,7$ using the same method as Comer but with a
2011 computer. Maddux failed to construct a Ramsey algebra for
$m=8$. In \cite{Kowalski}, Kowalski simultaneously and independently
derives results that match ours for $2\leq m \leq 120$. In addition,
he finds different constructions over finite fields of prime-power
order. The present authors independently rediscovered Comer's method
of using so-called cyclotomic classes, and we show that Comer's
method does not work for $m=8$, but does work for all $m$ between 9
and 400, except possibly for $m=13$. Therefore, cyclic Ramsey
algebras in $m$ colors exist for all $m$ between 2 and 400, except
possibly $8$ and $13$. In addition, we found some SSFCMBs that
failed to be Ramsey algebras (because they failed to satisfy the
mandatory triangle condition).


\section{Description of the search algorithm}\label{description}

To describe the algorithm we used to search for these SSFCMBs, we first bring the reader's attention to a property of the examples mentioned for $m = 2$ and $m = 3$. Recall that if $m = 2$, we may take $N = 5$, $X_{0} = \left\{1,4\right\}$, and $X_{1} = \left\{2,3\right\}$. Notice that $2 \in \mathbb{Z}_{5}$ is a generator of $\mathbb{Z}_{5}^{\times}$. Modulo 5, we have $X_{0} = \left\{2^{0},2^{2}\right\}$ and $X_{1} = \left\{2^{1},2^{3}\right\}$.

For $m = 3$, we had $N = 13$ with
\begin{eqnarray*}
X_{0} & = & \left\{1,5,8,12\right\},\\
X_{1} & = & \left\{2,3,10,11\right\},\ \text{and}\\
X_{2} & = & \left\{4,6,7,9\right\}.
\end{eqnarray*}
Again, $2$ is a generator of $\mathbb{Z}_{13}^{\times}$, and we also have

\begin{eqnarray*}
X_{0} & = & \left\{2^{0},2^{3},2^{6},2^{9}\right\},\\
X_{1} & = & \left\{2^{1},2^{4},2^{7},2^{10}\right\},\ \text{and}\\
X_{2} & = & \left\{2^{2},2^{5},2^{8},2^{11}\right\}.
\end{eqnarray*}

Based on these two constructions\footnote{These two examples are well-known  ``folklore" among relation-algebraists, and can be found in many sources.}, we tried to continue this pattern. That is, given $m$, we look at primes $N = mk + 1$ with $k$ even. We find a generator $x$ of $\mathbb{Z}_{N}^{\times}$, and construct the partition
\[X_{0} = \left\{x^{0},x^{m},x^{2m},\ldots,x^{(k-1)m}\right\}\]
with $X_{i} = x\cdot X_{i-1},\ \text{for $i = 1,2,\ldots,m-1$}$.

For ease, if a partition constructed in this fashion is an SSFCMB, we shall call it a \emph{single-generator SSFCMB}. We wish to make it clear that our algorithm  searches only for single-generator SSFCMBs.
\subsection{The search algorithm for single-generator SSFCMBs}
Below we describe the search algorithm as a series of steps.
\begin{enumerate}
\item
Use the Sieve of Eratosthenes to generate a list $P$ of primes
smaller than 2000000.
\item
Fix a positive integer $m$.
\item
Range over elements of $P$ until we come across a prime $N \equiv 1\ \text{(mod $2m$)}.$
\item
Set $k = \dfrac{N-1}{m}$.
\item
Find the prime divisors $p_{1},p_{2},\ldots,p_{r}$ of $N-1$.
\item
Find the smallest $x \in \mathbb{Z}_{N}^{\times}$ such that $x^{(N-1)/p_{i}} \not\equiv 1\ \text{(mod $N$)}$ for every $i \in \left\{1,2,\ldots,r\right\}$. Such $x$ is the smallest generator of the cyclic group $\mathbb{Z}_{N}^{\times}$.
\item
Compute $X_{0} = \left\{x^{0},x^{m},\ldots,x^{(k-1)m}\right\}$.
\item
Check that $X_{0}$ is sum-free and that $|X_{0}+X_{0}|=N-k$. If it
is, proceed; otherwise, discard $N$ and keep checking the elements
of $P$.
\item
Compute $X_{i} = x \cdot X_{i-1}$ for $i = 1,2,\ldots,m-1$.
\item
For $i = 1,2,\ldots,m-1$, check that $X_{0} + X_{i} = \mathbb{Z}_{N}\setminus \left\{0\right\}$.
\end{enumerate}

To see that this collection of steps is sufficient for the constructed partition to form an SSFCMB, we turn our attention to Section \ref{efficiency}, which provides the lemmas we used to complete some of the steps.

\section{Efficiency lemmas}\label{efficiency}

This section consists of a collection of lemmas which are used to improve the efficiency of the search algorithm. Together, they significantly reduce the number of checks that need to be made from what would be required in a na\"{\i}ve approach.

Lemma \ref{generatorlem} states that for given $m$ and prime $N \equiv 1\ \text{(mod $2m$)}$, it suffices to check only a single generator. Lemma \ref{sumfreelem} states that we  need only to check whether the element $1$ is in $X_{0} + X_{0}$ to determine if $X_{0}$ is sum-free. Lemma \ref{Xzerosumlem} states that if $X_{0}$ is a sum-free cyclic basis, then so is $X_{i}$ for $i = 1,\ldots,m-1$. Lemma \ref{trianglelem} reduces the number of calculations required to check if an SSFCMB satisfies the mandatory triangle condition from $O(N^{2})$ to $O(N)$.

Throughout this section, we let  $m \in \mathbb{Z}^{+}$ and let
 $N = mk + 1$ be a prime number.  Note that a version of Lemma \ref{generatorlem} appears in
\cite{comer83}.

\begin{lem}\label{generatorlem}
 If  $x$ and $y$ are generators of $\mathbb{Z}_{N}^{\times}$, then
\[\left\{x^{0},x^{m},x^{2m},\ldots,x^{(k-1)m}\right\} = \left\{y^{0},y^{m},y^{2m},\ldots,y^{(k-1)m}\right\}.\]
\end{lem}
\begin{proof}
Suppose $x$ and $y$ are generators of $\mathbb{Z}_{N}^{\times}$. We must show that every power of $y^{m}$ is some power of $x^{m}$.

To that end, fix a nonnegative integer $\ell$. Since $x$ is a generator of $\mathbb{Z}_{N}^{\times}$, there exists an integer $\alpha$ so that $x^{\alpha} = y$. Hence,
\[y^{\ell m} = \left(x^{\alpha}\right)^{\ell m} = x^{\alpha \ell m} = x^{\left(\alpha \ell \right)m},\]
as desired.
\end{proof}

\begin{lem}\label{sumfreelem}
 If $x$ is a generator of $\mathbb{Z}_{N}^{\times}$ and $X_{0} = \left\{x^{0},x^{m},x^{2m},\ldots,x^{(k-1)m}\right\},$ then $X_{0}$ is sum-free if and only if $1 \notin \left(X_{0} + X_{0}\right)$.

\end{lem}

\begin{proof}
It is clear that if $X_{0}$ is sum-free, then $1 \notin \left(X_{0} + X_{0}\right)$. For the other direction, suppose $X_{0}$ is not sum-free. This means there exist $\alpha$, $\beta$, and $\gamma$ so that
\begin{equation}\label{sumfreelemeq}
x^{m\alpha} + x^{m\beta} = x^{m\gamma}.
\end{equation}
 If $\min \left\{\alpha, \beta,\gamma\right\} = \alpha$, we may factor out $x^{m\alpha}$ from both sides of (\ref{sumfreelemeq}) to get $1 + x^{m\left(\beta-\alpha\right)} = x^{m\left(\gamma - \alpha\right)}$, or $1 = x^{m\left(\gamma-\alpha\right)} - x^{m\left(\beta-\alpha\right)}$, so $1 \in \left(X_{0} - X_{0}\right)$. Since $X_{0}$ is symmetric, $X_{0} - X_{0} = X_{0} + X_{0}$.

Similarly, if $\min \left\{\alpha,\beta,\gamma\right\} = \gamma$, then we may factor out $x^{m\gamma}$ from both sides of (\ref{sumfreelemeq}), and get that $1 \in \left(X_{0} + X_{0}\right)$.
\end{proof}

\begin{lem}\label{Xzerosumlem}
 Suppose $x$ is a generator of $\mathbb{Z}_{N}^{\times}$. For $i \in \left\{0,1,\ldots,m-1\right\}$, define

\[ X_{i} = \left\{x^{i},x^{m + i},x^{2m + i},\ldots,x^{(k-1)m + i}\right\}.\]
If $X_{0} + X_{0} = \mathbb{Z}_{N}\setminus X_{0}$, then
\[X_{i} + X_{i} = \mathbb{Z}_{N}\setminus X_{i}\]
for all $i \in \left\{1,2,\ldots,m-1\right\}$.
\end{lem}

\begin{proof}

First we check that each $X_i$  is sum-free. Every element of $X_{i}
+ X_{i}$ is of the form
\[x^{\alpha m + i} + x^{\beta m + i}\]
for some integers $\alpha$ and $\beta$. If there is an integer $q$ so that $x^{\alpha m + i} + x^{\beta m + i} = x^{qm + i}$, then by factoring out $x^{i}$, we have
\[x^{\alpha m} + x^{\beta m} = x^{qm},\]
which is a contradiction, as $X_{0} + X_{0} = \mathbb{Z}_{N}\setminus X_{0}$.

Suppose $z \in \mathbb{Z}_{N}\setminus X_{i}$.
Recall that $x$ is a generator of $\mathbb{Z}_{N}^{\times}$, so there exists an integer $k$ so that $z = x^{k}$. Since $x^{k}\notin X_{i}$, we have $x^{k - i} \notin X_{0}$. This means there exist integers $\alpha$ and $\beta$ so that
\[x^{\alpha m} + x^{\beta m} = x^{k-i}.\]
Multiplying both sides by $x^{i}$ achieves the desired result.
\end{proof}

\begin{lem}\label{trianglelem}
 Suppose $x$ is a generator of $\mathbb{Z}_{N}^{\times}$. For $i \in \left\{0,1,\ldots,m-1\right\}$, define

\[ X_{i} = \left\{x^{i},x^{m + i},x^{2m + i},\ldots,x^{(k-1)m + i}\right\}.\]
If $X_{0} + X_{i} = \mathbb{Z}_{N}\setminus\left\{0\right\}$ for all $i \in \left\{1,2,\ldots,m-1\right\}$, then
\[\forall i \forall j \left(i \neq j \rightarrow X_{i} + X_{j} = \mathbb{Z}_{N}\setminus\left\{0\right\}\right).\]
\end{lem}

\begin{proof}
Fix $i$ and $j$ with $i \neq j$. Without loss of generality, say $j > i$. Given a nonnegative integer $k$, we need to show that there exist integers $\alpha$ and $\beta$ so that
\[x^{\alpha m + i} + x^{\beta m + j} = x^{k}.\]

Since $X_{0} + X_{j-i} = \mathbb{Z}_{N} \setminus \left\{0\right\}$, there exist integers $\alpha$ and $\beta$ so that
\[x^{\alpha m} + x^{\beta m + (j - i)} = x^{k - i}.\]
Multiplying both sides by $x^{i}$ gives the desired result.
\end{proof}

\section{Future directions}\label{future}
Although we have found constructions for many values of $m$, we have
not gained any insight into any sort of pattern, as the sequence of
successive moduli is not even monotonic. If there is a pattern, it
currently eludes the authors.

The recursive upper bound from \cite{greenwoodgleason,smallramsey}
gives $R(3,3,3,3,3,3,3,3)\leq 109602$. By checking every candidate
prime up through this bound, we were able to determine that there is
no single-generator SSFCMB for $m = 8$.

\begin{thm}
Let $N \in \mathbb{Z}^{+}$. There does not exist a partition of
$\mathbb{Z}_{N}\setminus \left\{0\right\}$ into  $8$ parts that is a
single-generator SSFCMB.
\end{thm}
For the case of $m = 13$, the recursive bound is too large for the
computing power  available to the authors to rule out existence of a
single-generator SSFCMB. (The recursive upper bound is $\approx
1.69\cdot 10^{10}$.) However, if there is such a construction for $m
= 13$, the modulus $N$ must exceed 190997. Since this value is more
than 100 times the size of those moduli for other similarly small
values of $m$, we conjecture that there is no such partition for $m
=13$.

\begin{conj}
Let $N \in \mathbb{Z}^{+}$. There does not exist a partition  of
$\mathbb{Z}_{N}\setminus \left\{0\right\}$ into 13 parts that is a
single-generator SSFCMB.
\end{conj}

\section{Acknowledgements}
The authors would like to thank Jian Shen for invaluable
conversations regarding this topic.  We also thank Roger Maddux for
background information on Ramsey/Monk/Maddux algebras and for some
independent verification, as well as the two anonymous referees for
their helpful comments on the manuscript. Finally, we thank Tomasz
Kowalski for the independent verification and for sharing his
manuscript with us.

\appendix
\section{Table of $m$ and corresponding moduli}\label{mikesandnovembers}
Below we include tables containing the corresponding smallest
modulus $N$ for each value of $m$,  together with the smallest
generator $x$ of $\mathbb{Z}_{N}^{\times}$ needed to construct a
single-generator SSFCMB. Notice that $N = mk + 1$ for some positive
integer $k$ in every case. To reconstruct any of the partitions, set
$X_{0} = \left\{x^{0},x^{m},x^{2m},\ldots,x^{(k-1)m}\right\}$ and
$X_{i} = x\cdot X_{i-1}$ for $i = 1,2,\ldots,m-1$. Hence,
independent verification of any of the triples below is quite
straightforward.

As mentioned in Section \ref{future}, the values $m = 8$ and  $m = 13$ are missing from the table.

\begin{tabular}{|r|r|r|}
\hline
$m$ & $N$ & $x$\\
\hline
$2$ & $5$ & $2$\\
\hline
$3$ & $13$ & $2$\\
\hline
$4$ & $41$ & $6$\\
\hline
$5$ & $71$ & $7$\\
\hline
$6$ & $97$ & $5$\\
\hline
$7$ & $491$ & $2$\\
\hline
$9$ & $523$ & $2$\\
\hline
$10$ & $1181$ & $7$\\
\hline
$11$ & $947$ & $2$\\
\hline
$12$ & $769$ & $11$\\
\hline
$14$ & $1709$ & $3$\\
\hline
$15$ & $1291$ & $2$\\
\hline
$16$ & $1217$ & $3$\\
\hline
$17$ & $4013$ & $2$\\
\hline
$18$ & $2521$ & $17$\\
\hline
$19$ & $1901$ & $2$\\
\hline
$20$ & $2801$ & $3$\\
\hline
$21$ & $1933$ & $5$\\
\hline
$22$ & $3257$ & $3$\\
\hline
$23$ & $3221$ & $10$\\
\hline
$24$ & $4129$ & $13$\\
\hline
$25$ & $3701$ & $2$\\
\hline
$26$ & $4889$ & $3$\\
\hline
$27$ & $5563$ & $2$\\
\hline
$28$ & $8849$ & $3$\\
\hline
$29$ & $6323$ & $2$\\
\hline
$30$ & $5521$ & $11$\\
\hline
$31$ & $6263$ & $5$\\
\hline
$32$ & $5441$ & $3$\\
\hline
$33$ & $8779$ & $11$\\
\hline
$34$ & $7481$ & $6$\\
\hline
$35$ & $7841$ & $12$\\
\hline
$36$ & $10009$ & $11$\\
\hline
$37$ & $13469$ & $2$\\
\hline
$38$ & $12161$ & $3$\\
\hline
$39$ & $8971$ & $2$\\
\hline
$40$ & $14561$ & $6$\\
\hline
\end{tabular}\ \ \ \ \
\begin{tabular}{|r|r|r|}
\hline
$m$ & $N$ & $x$\\
\hline
$41$ & $13367$ & $5$\\
\hline
$42$ & $19993$ & $10$\\
\hline
$43$ & $14621$ & $2$\\
\hline
$44$ & $12497$ & $3$\\
\hline
$45$ & $14401$ & $11$\\
\hline
$46$ & $14537$ & $3$\\
\hline
$47$ & $20117$ & $2$\\
\hline
$48$ & $18913$ & $7$\\
\hline
$49$ & $22541$ & $3$\\
\hline
$50$ & $22901$ & $2$\\
\hline
$51$ & $19687$ & $5$\\
\hline
$52$ & $29537$ & $3$\\
\hline
$53$ & $26501$ & $2$\\
\hline
$54$ & $21493$ & $2$\\
\hline
$55$ & $23321$ & $3$\\
\hline
$56$ & $23297$ & $3$\\
\hline
$57$ & $21319$ & $14$\\
\hline
$58$ & $30509$ & $2$\\
\hline
$59$ & $28439$ & $11$\\
\hline
$60$ & $26041$ & $13$\\
\hline
$61$ & $45263$ & $5$\\
\hline
$62$ & $27281$ & $6$\\
\hline
$63$ & $30367$ & $5$\\
\hline
$64$ & $39041$ & $3$\\
\hline
$65$ & $37181$ & $2$\\
\hline
$66$ & $29569$ & $17$\\
\hline
$67$ & $38459$ & $2$\\
\hline
$68$ & $64601$ & $3$\\
\hline
$69$ & $31741$ & $6$\\
\hline
$70$ & $45641$ & $11$\\
\hline
$71$ & $36353$ & $3$\\
\hline
$72$ & $37441$ & $17$\\
\hline
$73$ & $44531$ & $2$\\
\hline
$74$ & $58313$ & $3$\\
\hline
$75$ & $48751$ & $3$\\
\hline
$76$ & $39521$ & $3$\\
\hline
$77$ & $70379$ & $6$\\
\hline
\end{tabular}\ \ \ \ \
\begin{tabular}{|r|r|r|}
\hline
$m$ & $N$ & $x$\\
\hline
$78$ & $53197$ & $2$\\
\hline
$79$ & $64781$ & $2$\\
\hline
$80$ & $53441$ & $3$\\
\hline
$81$ & $65287$ & $3$\\
\hline
$82$ & $64781$ & $2$\\
\hline
$83$ & $113213$ & $2$\\
\hline
$84$ & $76777$ & $5$\\
\hline
$85$ & $91121$ & $6$\\
\hline
$86$ & $80153$ & $3$\\
\hline
$87$ & $70123$ & $2$\\
\hline
$88$ & $67409$ & $3$\\
\hline
$89$ & $131543$ & $5$\\
\hline
$90$ & $74161$ & $7$\\
\hline
$91$ & $81173$ & $2$\\
\hline
$92$ & $80777$ & $3$\\
\hline
$93$ & $78307$ & $2$\\
\hline
$94$ & $70877$ & $2$\\
\hline
$95$ & $100511$ & $11$\\
\hline
$96$ & $136897$ & $5$\\
\hline
$97$ & $96419$ & $6$\\
\hline
$98$ & $105449$ & $6$\\
\hline
$99$ & $87517$ & $2$\\
\hline
$100$ & $95801$ & $3$\\
\hline
$101$ & $154127$ & $5$\\
\hline
$102$ & $95881$ & $13$\\
\hline
$103$ & $119687$ & $5$\\
\hline
$104$ & $131249$ & $3$\\
\hline
$105$ & $89671$ & $6$\\
\hline
$106$ & $144161$ & $3$\\
\hline
$107$ & $88811$ & $2$\\
\hline
$108$ & $122041$ & $7$\\
\hline
$109$ & $128621$ & $2$\\
\hline
$110$ & $122321$ & $6$\\
\hline
$111$ & $95461$ & $2$\\
\hline
$112$ & $122753$ & $3$\\
\hline
$113$ & $120233$ & $3$\\
\hline
$114$ & $98953$ & $10$\\
\hline
\end{tabular}
\newpage
\begin{tabular}{|r|r|r|}
\hline
$m$ & $N$ & $x$\\
\hline
$115$ & $115001$ & $3$\\
\hline
$116$ & $159617$ & $3$\\
\hline
$117$ & $118873$ & $5$\\
\hline
$118$ & $159773$ & $2$\\
\hline
$119$ & $166601$ & $6$\\
\hline
$120$ & $120721$ & $14$\\
\hline
$121$ & $176903$ & $5$\\
\hline
$122$ & $160553$ & $3$\\
\hline
$123$ & $145879$ & $13$\\
\hline
$124$ & $171617$ & $3$\\
\hline
$125$ & $121001$ & $6$\\
\hline
$126$ & $165817$ & $15$\\
\hline
$127$ & $182627$ & $2$\\
\hline
$128$ & $129281$ & $3$\\
\hline
$129$ & $142159$ & $6$\\
\hline
$130$ & $225941$ & $2$\\
\hline
$131$ & $208553$ & $3$\\
\hline
$132$ & $187441$ & $13$\\
\hline
$133$ & $173699$ & $2$\\
\hline
$134$ & $243077$ & $2$\\
\hline
$135$ & $197101$ & $2$\\
\hline
$136$ & $215153$ & $3$\\
\hline
$137$ & $190979$ & $6$\\
\hline
$138$ & $156217$ & $5$\\
\hline
$139$ & $179033$ & $3$\\
\hline
$140$ & $191801$ & $3$\\
\hline
$141$ & $224473$ & $10$\\
\hline
$142$ & $218681$ & $13$\\
\hline
$143$ & $200201$ & $3$\\
\hline
$144$ & $184321$ & $13$\\
\hline
$145$ & $218081$ & $6$\\
\hline
$146$ & $257837$ & $2$\\
\hline
$147$ & $221677$ & $2$\\
\hline
$148$ & $262553$ & $3$\\
\hline
$149$ & $238103$ & $5$\\
\hline
$150$ & $199501$ & $2$\\
\hline
$151$ & $237977$ & $3$\\
\hline
\end{tabular}\ \ \ \ \
\begin{tabular}{|r|r|r|}
\hline
$m$ & $N$ & $x$\\
\hline
$152$ & $213713$ & $3$\\
\hline
$153$ & $245719$ & $11$\\
\hline
$154$ & $590129$ & $3$\\
\hline
$155$ & $220721$ & $3$\\
\hline
$156$ & $254281$ & $7$\\
\hline
$157$ & $282287$ & $5$\\
\hline
$158$ & $352973$ & $2$\\
\hline
$159$ & $246769$ & $7$\\
\hline
$160$ & $281921$ & $3$\\
\hline
$161$ & $303647$ & $7$\\
\hline
$162$ & $347329$ & $7$\\
\hline
$163$ & $240263$ & $5$\\
\hline
$164$ & $278801$ & $3$\\
\hline
$165$ & $266641$ & $19$\\
\hline
$166$ & $292493$ & $3$\\
\hline
$167$ & $313961$ & $3$\\
\hline
$168$ & $294673$ & $5$\\
\hline
$169$ & $277499$ & $2$\\
\hline
$170$ & $329801$ & $3$\\
\hline
$171$ & $302329$ & $7$\\
\hline
$172$ & $320609$ & $3$\\
\hline
$173$ & $330431$ & $23$\\
\hline
$174$ & $285709$ & $2$\\
\hline
$175$ & $449051$ & $2$\\
\hline
$176$ & $375233$ & $3$\\
\hline
$177$ & $355063$ & $7$\\
\hline
$178$ & $395873$ & $3$\\
\hline
$179$ & $307523$ & $2$\\
\hline
$180$ & $361441$ & $13$\\
\hline
$181$ & $381911$ & $17$\\
\hline
$182$ & $347621$ & $3$\\
\hline
$183$ & $345139$ & $2$\\
\hline
$184$ & $315377$ & $3$\\
\hline
$185$ & $383321$ & $3$\\
\hline
$186$ & $418129$ & $7$\\
\hline
$187$ & $394571$ & $6$\\
\hline
$188$ & $429017$ & $3$\\
\hline
\end{tabular}\ \ \ \ \
\begin{tabular}{|r|r|r|}
\hline
$m$ & $N$ & $x$\\
\hline
$189$ & $333019$ & $11$\\
\hline
$190$ & $339341$ & $2$\\
\hline
$191$ & $557339$ & $2$\\
\hline
$192$ & $467329$ & $23$\\
\hline
$193$ & $452393$ & $3$\\
\hline
$194$ & $484613$ & $2$\\
\hline
$195$ & $280411$ & $2$\\
\hline
$196$ & $502937$ & $5$\\
\hline
$197$ & $397547$ & $2$\\
\hline
$198$ & $410257$ & $5$\\
\hline
$199$ & $342281$ & $3$\\
\hline
$200$ & $479201$ & $3$\\
\hline
$201$ & $412051$ & $2$\\
\hline
$202$ & $617717$ & $3$\\
\hline
$203$ & $426707$ & $2$\\
\hline
$204$ & $374137$ & $5$\\
\hline
$205$ & $345221$ & $3$\\
\hline
$206$ & $446609$ & $3$\\
\hline
$207$ & $424351$ & $3$\\
\hline
$208$ & $421409$ & $3$\\
\hline
$209$ & $390413$ & $2$\\
\hline
$210$ & $475441$ & $13$\\
\hline
$211$ & $505979$ & $2$\\
\hline
$212$ & $632609$ & $3$\\
\hline
$213$ & $569137$ & $5$\\
\hline
$214$ & $582509$ & $2$\\
\hline
$215$ & $484181$ & $2$\\
\hline
$216$ & $565489$ & $13$\\
\hline
$217$ & $521669$ & $2$\\
\hline
$218$ & $513173$ & $2$\\
\hline
$219$ & $536989$ & $7$\\
\hline
$220$ & $531521$ & $6$\\
\hline
$221$ & $687311$ & $7$\\
\hline
$222$ & $461317$ & $2$\\
\hline
$223$ & $516023$ & $5$\\
\hline
$224$ & $525953$ & $3$\\
\hline
$225$ & $601201$ & $14$\\
\hline
\end{tabular}
\newpage
\begin{tabular}{|r|r|r|}
\hline
$m$ & $N$ & $x$\\
\hline
$226$ & $539237$ & $2$\\
\hline
$227$ & $523463$ & $5$\\
\hline
$228$ & $585049$ & $7$\\
\hline
$229$ & $583493$ & $2$\\
\hline
$230$ & $555221$ & $10$\\
\hline
$231$ & $609379$ & $2$\\
\hline
$232$ & $609233$ & $3$\\
\hline
$233$ & $642149$ & $3$\\
\hline
$234$ & $496549$ & $2$\\
\hline
$235$ & $635441$ & $12$\\
\hline
$236$ & $575369$ & $3$\\
\hline
$237$ & $501493$ & $2$\\
\hline
$238$ & $637841$ & $21$\\
\hline
$239$ & $664421$ & $2$\\
\hline
$240$ & $653281$ & $7$\\
\hline
$241$ & $603947$ & $2$\\
\hline
$242$ & $691637$ & $2$\\
\hline
$243$ & $618679$ & $3$\\
\hline
$244$ & $746153$ & $3$\\
\hline
$245$ & $623771$ & $2$\\
\hline
$246$ & $661741$ & $2$\\
\hline
$247$ & $736061$ & $2$\\
\hline
$248$ & $631409$ & $3$\\
\hline
$249$ & $761443$ & $2$\\
\hline
$250$ & $653501$ & $2$\\
\hline
$251$ & $646577$ & $3$\\
\hline
$252$ & $632521$ & $11$\\
\hline
$253$ & $719027$ & $5$\\
\hline
$254$ & $689357$ & $2$\\
\hline
$255$ & $632911$ & $6$\\
\hline
$256$ & $724481$ & $3$\\
\hline
$257$ & $668201$ & $6$\\
\hline
$258$ & $751297$ & $5$\\
\hline
$259$ & $746957$ & $2$\\
\hline
$260$ & $710321$ & $3$\\
\hline
$261$ & $694261$ & $2$\\
\hline
$262$ & $793337$ & $3$\\
\hline
$263$ & $803729$ & $3$\\
\hline
\end{tabular}\ \ \ \ \
\begin{tabular}{|r|r|r|}
\hline
$m$ & $N$ & $x$\\
\hline
$264$ & $699073$ & $5$\\
\hline
$265$ & $880331$ & $7$\\
\hline
$266$ & $1229453$ & $2$\\
\hline
$267$ & $690997$ & $2$\\
\hline
$268$ & $941753$ & $3$\\
\hline
$269$ & $833363$ & $2$\\
\hline
$270$ & $689581$ & $10$\\
\hline
$271$ & $804329$ & $3$\\
\hline
$272$ & $875297$ & $3$\\
\hline
$273$ & $716899$ & $3$\\
\hline
$274$ & $778709$ & $2$\\
\hline
$275$ & $929501$ & $3$\\
\hline
$276$ & $724777$ & $10$\\
\hline
$277$ & $916871$ & $7$\\
\hline
$278$ & $856241$ & $3$\\
\hline
$279$ & $921259$ & $2$\\
\hline
$280$ & $975521$ & $11$\\
\hline
$281$ & $911003$ & $2$\\
\hline
$282$ & $680749$ & $2$\\
\hline
$283$ & $946919$ & $7$\\
\hline
$284$ & $983777$ & $3$\\
\hline
$285$ & $949621$ & $10$\\
\hline
$286$ & $1035893$ & $2$\\
\hline
$287$ & $1080269$ & $2$\\
\hline
$288$ & $816769$ & $13$\\
\hline
$289$ & $826541$ & $2$\\
\hline
$290$ & $1006301$ & $2$\\
\hline
$291$ & $1230349$ & $2$\\
\hline
$292$ & $1073393$ & $3 $\\
\hline
$293$ & $1181963$ & $2$\\
\hline
$294$ & $981373$ & $6$\\
\hline
$295$ & $918041$ & $3$\\
\hline
$296$ & $877937$ & $3$\\
\hline
$297$ & $880903$ & $3$\\
\hline
$298$ & $1086509$ & $2$\\
\hline
$299$ & $1288691$ & $2$\\
\hline
$300$ & $940801$ & $41$\\
\hline
$301$ & $1104671$ & $7$\\
\hline
\end{tabular}\ \ \ \ \
\begin{tabular}{|r|r|r|}
\hline
$m$ & $N$ & $x$\\
\hline
$302$ & $1111361$ & $3$\\
\hline
$303$ & $948391$ & $30$\\
\hline
$304$ & $964289$ & $3$\\
\hline
$305$ & $1087631$ & $34$\\
\hline
$306$ & $1171981$ & $2$\\
\hline
$307$ & $925913$ & $3$\\
\hline
$308$ & $1153769$ & $3$\\
\hline
$309$ & $975823$ & $3$\\
\hline
$310$ & $1009361$ & $3$\\
\hline
$311$ & $1015727$ & $5$\\
\hline
$312$ & $1129441$ & $14$\\
\hline
$313$ & $1214441$ & $3$\\
\hline
$314$ & $1366529$ & $3$\\
\hline
$315$ & $1167391$ & $14$\\
\hline
$316$ & $1216601$ & $6$\\
\hline
$317$ & $1381487$ & $5$\\
\hline
$318$ & $1176601$ & $11$\\
\hline
$319$ & $1052063$ & $5$\\
\hline
$320$ & $1210241$ & $3$\\
\hline
$321$ & $1145329$ & $7$\\
\hline
$322$ & $1409717$ & $2$\\
\hline
$323$ & $1149881$ & $7$\\
\hline
$324$ & $1082161$ & $7$\\
\hline
$325$ & $1066001$ & $3$\\
\hline
$326$ & $1270097$ & $3$\\
\hline
$327$ & $1043131$ & $11$\\
\hline
$328$ & $1144721$ & $3$\\
\hline
$329$ & $1309421$ & $10$\\
\hline
$330$ & $1151041$ & $17$\\
\hline
$331$ & $1397483$ & $2$\\
\hline
$332$ & $1496657$ & $3$\\
\hline
$333$ & $1235431$ & $3$\\
\hline
$334$ & $1269869$ & $2$\\
\hline
$335$ & $1345361$ & $6$\\
\hline
$336$ & $1109473$ & $5$\\
\hline
$337$ & $1317671$ & $11$\\
\hline
$338$ & $1643357$ & $2$\\
\hline
$339$ & $1332949$ & $6$\\
\hline
\end{tabular}\ \ \ \ \
\newpage
\begin{tabular}{|r|r|r|}
\hline
$m$ & $N$ & $x$\\
\hline
$340$ & $1247801$ & $3$\\
\hline
$341$ & $1434929$ & $3$\\
\hline
$342$ & $1240777$ & $7$\\
\hline
$343$ & $1423451$ & $2$\\
\hline
$344$ & $1922273$ & $3$\\
\hline
$345$ & $1267531$ & $2$\\
\hline
$346$ & $1325873$ & $3$\\
\hline
$347$ & $1345667$ & $2$\\
\hline
$348$ & $1251409$ & $14$\\
\hline
$349$ & $1741511$ & $7$\\
\hline
$350$ & $1378301$ & $10$\\
\hline
$351$ & $1308529$ & $7$\\
\hline
$352$ & $1490369$ & $3$\\
\hline
$353$ & $1650629$ & $2$\\
\hline
$354$ & $1215637$ & $2$\\
\hline
$355$ & $1392311$ & $13$\\
\hline
$356$ & $1536497$ & $3$\\
\hline
$357$ & $1391587$ & $2$\\
\hline
$358$ & $1644653$ & $2$\\
\hline
$359$ & $1482671$ & $7$\\
\hline
$360$ & $1204561$ & $29$\\
\hline
$361$ & $1608617$ & $3$\\
\hline
$362$ & $1755701$ & $2$\\
\hline
$363$ & $1577599$ & $3$\\
\hline
$364$ & $1486577$ & $3$\\
\hline
$365$ & $1658561$ & $6$\\
\hline
$366$ & $1630897$ & $10$\\
\hline
$367$ & $1551677$ & $2$\\
\hline
$368$ & $1389569$ & $3$\\
\hline
$369$ & $1461979$ & $2$\\
\hline
\end{tabular}\ \ \ \ \
\begin{tabular}{|r|r|r|}
\hline
$m$ & $N$ & $x$\\
\hline
$370$ & $1400081$ & $3$\\
\hline
$371$ & $1570073$ & $3$\\
\hline
$372$ & $1490233$ & $7$\\
\hline
$373$ & $2387201$ & $3$\\
\hline
$374$ & $1831853$ & $2$\\
\hline
$375$ & $1695751$ & $3$\\
\hline
$376$ & $1711553$ & $3$\\
\hline
$377$ & $1627133$ & $2$\\
\hline
$378$ & $1751653$ & $2$\\
\hline
$379$ & $1686551$ & $11$\\
\hline
$380$ & $1931921$ & $3$\\
\hline
$381$ & $1423417$ & $11$\\
\hline
$382$ & $1642601$ & $3$\\
\hline
$383$ & $1607069$ & $2$\\
\hline
$384$ & $1545217$ & $15$\\
\hline
$385$ & $1657811$ & $2$\\
\hline
$386$ & $1818833$ & $3$\\
\hline
$387$ & $1963639$ & $3$\\
\hline
$388$ & $1689353$ & $3$\\
\hline
$389$ & $2059367$ & $5$\\
\hline
$390$ & $1861861$ & $2$\\
\hline
$391$ & $1730567$ & $5$\\
\hline
$392$ & $1821233$ & $3$\\
\hline
$393$ & $1758283$ & $3$\\
\hline
$394$ & $1795853$ & $2$\\
\hline
$395$ & $1837541$ & $3$\\
\hline
$396$ & $1744777$ & $7$\\
\hline
$397$ & $1971503$ & $5$\\
\hline
$398$ & $2173877$ & $2$\\
\hline
$399$ & $2108317$ & $2$\\
\hline
$400$ & $1772801$ & $3$\\
\hline

\end{tabular}


\begin{thebibliography}{1}

\bibitem{comer83}
S.~D. Comer.
\newblock Color schemes forbidding monochrome triangles.
\newblock In {\em Proceedings of the fourteenth Southeastern conference on
  combinatorics, graph theory and computing (Boca Raton, Fla., 1983)},
  volume~39, pages 231--236, 1983.

\bibitem{greenwoodgleason}
R.~E. Greenwood and A.~M. Gleason.
\newblock Combinatorial relations and chromatic graphs.
\newblock {\em Canad. J. Math.}, 7:1--7, 1955.

\bibitem{HH}
R.~Hirsch and I.~Hodkinson.
\newblock {\em Relation algebras by games}, volume 147 of {\em Studies in Logic
  and the Foundations of Mathematics}.
\newblock North-Holland Publishing Co., Amsterdam, 2002.
\newblock With a foreword by Wilfrid Hodges.

\bibitem{Jia10}
X.~Jia.
\newblock On the exact order of asymptotic bases and bases for finite cyclic
  groups.
\newblock In {\em Additive number theory}, pages 179--193. Springer, New York,
  2010.

\bibitem{JiaShen}
X.~Jia and J.~Shen.
\newblock Extremal Bases for Finite Cyclic Groups.
\newblock Preprint,  2012.

\bibitem{Kowalski}
T.~Kowalski.
\newblock Representability of {R}amsey relation algebras.
\newblock {\em Algebra Universalis}, to appear.

\bibitem{M}
R.~Maddux. Personal communication.


\bibitem{Mad11}
R.~Maddux.
\newblock Do all the {R}amsey algebras exist?
\newblock Presented at the AMS sectional meeting in Iowa City on March 18,
  2011.

\bibitem{Mad82}
R.~Maddux.
\newblock Some varieties containing relation algebras.
\newblock {\em Trans. Amer. Math. Soc.}, 272(2):501--526, 1982.

\bibitem{smallramsey}
S.~P. Radziszowski.
\newblock Small {R}amsey numbers.
\newblock {\em Electron. J. Combin.}, 1:Dynamic Survey 1, 30 pp. (electronic),
  1994.

\end{thebibliography}

\end{document}